\documentclass[12pt]{amsart}
\usepackage{amscd,amsmath,amsthm,amssymb,graphics}
\usepackage{slashbox}

\usepackage{pgf,tikz, pifont}

\usetikzlibrary{arrows}

\usepackage[all]{xy}
\usepackage{lineno}

\def\frk{\frak}               

\def\Phi{{\frk n}}
\def\Phi{{\frk N}}
\def\opn#1#2{\def#1{\operatorname{#2}}} 
\opn\chara{char} \opn\length{\ell} \opn\pd{pd} \opn\rk{rk}
\opn\projdim{proj\,dim} \opn\injdim{inj\,dim} \opn\rank{rank}
\opn\depth{depth} \opn\grade{grade} \opn\height{height}
\opn\embdim{emb\,dim} \opn\codim{codim}

\opn\Tr{Tr} \opn\bigrank{big\,rank}
\opn\superheight{superheight}\opn\lcm{lcm}
\opn\trdeg{tr\,deg}
\opn\reg{reg} \opn\lreg{lreg} \opn\ini{in} \opn\lpd{lpd}
\opn\size{size}\opn\bigsize{bigsize}
\opn\cosize{cosize}\opn\bigcosize{bigcosize}
\opn\sdepth{sdepth}\opn\sreg{sreg}
\opn\link{link}\opn\fdepth{fdepth}
\opn\index{index}
\opn\index{index}
\opn\indeg{indeg}
\opn\N{N}
\opn\SSC{SSC}
\opn\SC{SC}
\opn\lk{lk}
\opn\div{div} \opn\Div{Div} \opn\cl{cl} \opn\Cl{Cl}
\opn\Spec{Spec} \opn\Supp{Supp} \opn\supp{supp} \opn\Sing{Sing}
\opn\Ass{Ass} \opn\Min{Min}\opn\Mon{Mon} \opn\dstab{dstab} \opn\astab{astab}
\opn\Syz{Syz}
\opn\reg{reg}
\opn\Ann{Ann} \opn\Rad{Rad} \opn\Soc{Soc}
\opn\Im{Im} \opn\Ker{Ker} \opn\Coker{Coker} \opn\Am{Am}
\opn\Hom{Hom} \opn\Tor{Tor} \opn\Ext{Ext} \opn\End{End}\opn\Der{Der}
\opn\Aut{Aut} \opn\id{id}

\opn\nat{nat}
\opn\pff{pf}
\opn\Pf{Pf} \opn\GL{GL} \opn\SL{SL} \opn\mod{mod} \opn\ord{ord}
\opn\Gin{Gin} \opn\Hilb{Hilb}\opn\sort{sort}
\opn\initial{init}
\opn\ende{end}
\opn\height{height}
\opn\type{type}
\opn\aff{aff} \opn\con{conv} \opn\relint{relint} \opn\st{st}
\opn\lk{lk} \opn\cn{cn} \opn\core{core} \opn\vol{vol}
\opn\link{link} \opn\Link{Link}\opn\lex{lex}
\opn\gr{gr}

\def\pot#1#2{#1[\kern-0.28ex[#2]\kern-0.28ex]}

\opn\dirlim{\underrightarrow{\lim}}
\opn\inivlim{\underleftarrow{\lim}}
\def\Implies{\ifmmode\Longrightarrow \else
        \unskip${}\Longrightarrow{}$\ignorespaces\fi}
\def\implies{\ifmmode\Rightarrow \else
        \unskip${}\Rightarrow{}$\ignorespaces\fi}
\def\iff{\ifmmode\Longleftrightarrow \else
        \unskip${}\Longleftrightarrow{}$\ignorespaces\fi}

\let\:=\colon
\newtheorem{Theorem}{Theorem}[section]
 \newtheorem{Lemma}[Theorem]{Lemma}
 \newtheorem{Corollary}[Theorem]{Corollary}
 \newtheorem{Proposition}[Theorem]{Proposition}

 \newtheorem{Example}[Theorem]{Example}
 
 \newtheorem{Definition}[Theorem]{Definition}

\newtheorem{Notation}[Theorem]{Notation}

\let\epsilon\varepsilon
\let\kappa=\varkappa
\def\qed{\ifhmode\textqed\fi
      \ifmmode\ifinner\quad\qedsymbol\else\dispqed\fi\fi}
\def\textqed{\unskip\nobreak\penalty50
       \hskip2em\hbox{}\nobreak\hfil\qedsymbol
       \parfillskip=0pt \finalhyphendemerits=0}
\def\dispqed{\rlap{\qquad\qedsymbol}}

\opn\dis{dis}
\def\pnt{{\raise0.5mm\hbox{\large\bf.}}}

\opn\Lex{Lex}

\begin{document}

\title{ Linear resolutions and quasi-linearity of monomial ideals}
\author{Dancheng Lu}

\address{School  of Mathematical Sciences, Soochow University, 215006 Suzhou, P.R.China}
\email{ludancheng@suda.edu.cn}



\begin{abstract} We introduce the notion of quasi-linearity and prove it is necessary  for a monomial ideal to have a linear resolution and clarify all the quasi-linear   monomial  ideals generated in degree 2. We also introduce  the notion of  a strongly linear monomial over a monomial ideal and prove that if $\mathbf{u}$ is a monomial strongly linear over $I$ then $I$  has a linear resolution (resp: is quasi-linear) if and only if $I+\mathbf{u}\mathfrak{p}$ has   a linear resolution (resp: is quasi-linear). Here $\mathfrak{p}$ is any monomial prime ideal.
\end{abstract}

\maketitle

\section*{Introduction}

Graded free resolutions  of graded modules  have been a central topic in commutative  algebra for a long history. Let $M$  be a finitely generated graded module over a polynomial ring $\mathbb{K}[x_1,\cdots, x_n]$, where $\mathbb{K}$ is a field. Then $M$ has a minimal graded free resolution of length at most $n$. This is the celebrated Hilbert's Syzygy Theorem. However, there is no explicit description of   the   minimal free resolution for an arbitrary graded module or even an arbitrary  (quadratic) monomial ideal till now, see \cite[Section 6.3]{I10}.

Recall that  a linear  resolution is a graded free resolution in which  the matrices of all  differential maps  have  entries in the set of linear forms. The graded ideals with linear resolutions have been investigated  quite extensively, for it is  pertinent to the Koszul property of algebras.  An important question in this area is how to test if a graded module or  a monomial ideal has a linear resolution. It  was  proved in \cite{HT02} that if $I$ is a graded ideal generated in a single degree with linear quotients, then $I$ has  a linear resolution.  This sufficient condition is independent of character of $\mathbb{K}$ and is easily checked relatively.  On the other side, the question when a quadratic  monomial ideal  has a  linear resolution was fully resolved:   \cite{Fro90} addressed  the squarefree case  and  \cite{HHZ04}  did  the general case.

In  this paper, we  investigate further  properties of monomial ideals with  a linear resolution. In Section 2, we first prove if a monomial ideal  $I$ has a linear resolution then it is {\it quasi-linear}, see Theorem~\ref{main}.  By definition, a monomial ideal $I$ is quasi-linear if  the colon ideal $(G(I)\setminus \{\mathbf{u}\}):\mathbf{u}$ is generated by linear forms for every $\mathbf{u}\in G(I)$.  This result is based on the main result of \cite{DDL}  together with the tools of  Alexander duality and polarization. Thus, we have the following hierarchical relationships for a monomial ideal $I$ generated in a single degree:\begin{equation*}
\begin{aligned}I \mbox{ has linear quotients} \Longrightarrow I \mbox{ has a linear resolution for any field } \mathbb{K} \\ \Longrightarrow I\mbox{ has a linear resolution for some field } \mathbb{K} \Longrightarrow I\mbox{ is quasi-linear}.\end{aligned}
\end{equation*}

A direct consequence of the above result is that a monomial ideal generated in a single degree with linear quotients is characterized in terms of the existences of a sequence of monomial ideals with a linear resolution, see Corollary~\ref{inter}. We say  a monomial ideal $I$ is  {\it critical linear} if it has a linear resolution, but $(G(I)\setminus \{\mathbf{u}\})$ has not a linear resolution for any $\mathbf{u}\in G(I)$.
 Another  consequence is  that every  monomial ideal with a linear resolution  is linear over a critical linear monomial. Here, for monomial  ideals $I$ and $J$, we say $I$ is {\it linear over} $J$ if there exist $r>0$ and monomials $\mathbf{u}_1,\ldots,\mathbf{u}_r$ such that $I=(J, \mathbf{u}_1,\ldots,\mathbf{u}_r)$ and $(J,\mathbf{u}_1,\ldots,\mathbf{u}_{i-1}):\mathbf{u}_i$ is generated by linear forms for $i=1,\ldots,r$.

It is an  interesting question  when  a quadratic monomial ideal is quasi-linear. We prove in Proposition~\ref{quadratic} that if $I$ is a monomial ideal generated in degree 2 then $I$  is quasi-linear if and only if the induced matching number of $G_I$ is equal to 1.  Here $G_I$ is a simple graph associated to $I$. In addition, if $\mathbf{u}$ is a monomial of degree $d$, we prove that $(G(\mathfrak{m}^d)\setminus \{\mathbf{u}\})$ is quasi-linear if and only if $|\mathrm{supp}(\mathbf{u})|\neq 2$.  Here $\mathfrak{m}$ denotes the maximal monomial ideal $(x_1,\ldots,x_n)$.

Let $I$ and $J$ be  monomial ideals generated in degree $d$. If $I$ is linear over $J$, then we can deduce that $I$ has a linear resolution from the condition that $J$ has a linear resolution.  However, the converse is not true. So it is natural to ask the following question:

\begin{center}

 {\it \qquad Is there  a  notion stronger  than ``linear over" such that if $I$  is ``strongly linear over" $J$, then $I$  has a linear resolution if and only if $J$ has a linear resolution?  }

 \end{center}

 In Section 3, we answer this question positively by introducing   the notion of a {\it monomial strongly linear} over a monomial ideal.  Let $I$ be a monomial ideal generated in degree $d$. We call a monomial $\mathbf{u}$ to be {\it  strongly linear }  over $I$ if $I:\mathbf{u}$ is generated by variables and $\deg(\mathbf{u})=d-1$. In this case, we say the ideal  $I+\mathbf{u}\mathfrak{p}$ is
{\it (1-step) strongly linear} over $I$, where $\mathfrak{p}$ is any monomial prime ideal. Strongly linear monomials  over a monomial ideal are not uncommon.
\begin{Example}\em  Let $I=(x_1x_3^2,,x_3^2x_4,x_3x_4^2)$. Then $x_3x_4$ is a  monomial strongly linear over $I$; If $I=(x_1x_2x_3, x_3^2x_4,x_3x_4^2)$, then $x_1x_2$ is a monomial  strongly linear over $I$.
\end{Example}

\begin{Example} \em
If $J$ is a stable monomial ideal in   degree $d$, then $\mathfrak{m}^d$ is strongly linear over $J$, see Proposition~\ref{stable}.
\end{Example}

   Suppose that $I$ is strongly linear over $J$.  We prove in Theorem~\ref{main1}  that both  $I$ and $J$ have the same Betti numbers in their non-linear strands, i.e.,  $\beta_{i,i+j}(I)=\beta_{i,i+j}(J)$ for all $i$ and all $j\neq d$. In particular, both $I$ and $J$ share the  same regularity, and so $I$ has a  linear resolution if and only if $J$ has a linear resolution. Under the same assumption,   it is proved  in  Corollary~\ref{firstmain} that $I$ is quasi-linear if and only if $J$ is quasi-linear. Finally, if $I$ is 1-step  strongly linear over $J$, we compute the difference between the Betti numbers in the linear strands of $I$ and $J$.

The concept of a  {\it simplicial maximal subcircuit}  was introduced in \cite{BYZ} to define chordal clutters  and was further investigated in \cite{BHY,BY}.  In the last section, we show  that if a $(d-1)-$subset $e$ of $[n]$ is a  simplicial maximal subcircuit of a $d$-uniform clutter $\mathcal{C}$, then  $\mathbf{x}_e:=\prod_{i\in e}x_i$ is  strongly linear over $I(\overline{\mathcal{C}})$. Thanks to this observation, we recover  several of  main results of  \cite{BY}.  More explicitly,  \cite[Theorem 2.4 and Corollary 2.6]{BY} are the direct consequences of Proposition~\ref{B}. It turns out that the proofs of our results are  much simpler than the proofs of the corresponding results given in \cite{BY}.

\section{Preliminaries}

In this section, we fix notation and recall some concepts and results which are useful  in this paper.

Throughout this paper,    we denote by $[n]$ the finite set $\{1,2,\ldots,n\}$ and let $R=\mathbb{\mathbb{K}}[x_1,\ldots,x_n]$ be the polynomial ring in variables $x_1,\ldots,x_n$ over a field $\mathbb{K}$. The ring $R$ is naturally graded with $\deg(x_i)=1$ for $i=1,\ldots, n$. Sometimes, we consider $R$ as a  multi-graded ring with $\mathrm{mdeg}(x_i)=e_i$, where $e_1,\ldots,e_n$  are the standard basis of $\mathbb{Z}^n$. If  $\mathbf{u}=x_1^{a_1}\cdots x_n^{a_n}\in R$,  then $\deg(\mathbf{u})=a_1+\cdots+a_n$ and $\mathrm{mdeg}(\mathbf{u})=(a_1,\ldots,a_n)$.  Let $\mathbf{a}$ be a vector $(a_1,\ldots,a_n)\in \mathbb{Z}_{\geq 0}^n$.  We also use $\mathbf{x}^{\mathbf{a}}$ to denote the monomial $x_1^{a_1}\cdots x_n^{a_n}$, and use $\mathrm{supp}(\mathbf{a})$ or $\mathrm{supp}(\mathbf{x}^{\mathbf{a}})$ to denote the set $\{i\in [n]\:\; a_i\neq 0\}$.  If $A$ is  a subset of $[n]$, then  $|A|$ denotes  the number of elements of $A$, and $\begin{pmatrix}A\\i\end{pmatrix}$   denotes  the collection of all $i$-subsets of $A$ for $1\leq i\leq |A|$.

\subsection{Multi-graded free resolutions and Betti numbers} Let $M$ be a finitely generated multi-graded $R$-module. Then $M$ admits a minimal multi-graded free resolution: $$0\rightarrow F_p\rightarrow F_{p-1}\rightarrow \cdots \rightarrow F_0\rightarrow M\rightarrow 0.$$
If we write  $F_i\cong \bigoplus_{\mathbf{a}\in \mathbb{Z}^n}R[-\mathbf{a}]^{\beta_{i,\mathbf{a}}(M)}$, then the numbers $\beta_{i,\mathbf{a}}(M)=\dim_{\mathbb{K}} \mathrm{Tor}_i^R(M,\mathbb{K})_{\mathbf{a}}$ are called the {\it multi-graded Betti numbers} of $M$. The graded Betti numbers are defined as follows:
$$\beta_{i,j}(M)=\sum_{\mathbf{a}\in \mathbb{Z}^n, |\mathbf{a}|=j}\beta_{i,\mathbf{a}}(M).$$

The regularity of a graded module $M$, which measures the complexity of its graded free resolution,  is a very important invariant. It is defined to be the number
$$\reg(M)=\max\{j-i\:\; \beta_{i,j}(M)\neq 0\}.$$

Let $I$ be a monomial ideal of $R$. Then $I$ and $R/I$ are naturally multi-graded $R$-modules. Moreover, we have the following easy formula: $$\beta_{i,\mathbf{a}}(I)=\beta_{i+1,\mathbf{a}}(R/I) \mbox{ for all } i\geq 0 \mbox{ and } \mathbf{a}\in \mathbb{Z}^n ;$$ and
$$ \reg(R/I)=\reg(I)-1.$$
\subsection{Linear resolution}

Let $d\in \mathbb{Z}$. A  finitely generated graded module $M$  has a $d$-{\it linear resolution} if $\beta_{i,j}(M)=0$ for any pair $i,j$ with $j-i\neq d$. We record the following well-known result for a better understanding of linear resolutions.
\begin{Proposition} Let $M$ be a finitely generated graded $R$-module. Then the following are equivalent:
\begin{enumerate}
                                                                                                           \item $M$ has a $d$-linear resolution;
                                                                                                           \item $M$ is generated in degree $d$ and $\reg(M)=d$;
                                                                                                           \item The matrices of all  differential maps in a minimal free resolution of $M$  have  entries in the set of linear forms.
                                                                                                         \end{enumerate}
\end{Proposition}
A graded  ideal $I$ has {\it linear quotients} if there is a system of minimal homogeneous generators $f_1,f_2,\ldots,f_m$ such that the colon ideal $(f_1,\ldots,f_{i-1}):f_i$ is generated by linear forms for each $i=2,\ldots,m$. By e.g. \cite[Proposition 8.2.1]{HH}, if $I$ is generated in degree $d$ and $I$ has linear quotients then $I$ has a $d$-linear resolution.

\subsection{Monomial ideals} Let $I$ be a monomial ideal of $R$. We denote by $G(I)$  the unique minimal set of   monomial generators of $I$.  Given monomials $\mathbf{u}$ and $\mathbf{v}$, we denote by $[\mathbf{u},\mathbf{v}]$ and $(\mathbf{u},\mathbf{v})$  the least common multiple and the greatest common divisor of $\mathbf{u}$ and $\mathbf{v}$, respectively.
Let $J$  be an another monomial ideal. The colon ideal $I:J$ is defined to be the ideal $\{f\in R\:\; fg\in I \mbox{ for any }g\in J\}$. We usually write $I:\mathbf{v}$ for $I:(\mathbf{v})$ for short.
It is well-known that
$$I:\mathbf{v}=(\frac{[\mathbf{u},\mathbf{v}]}{\mathbf{v}}(=\frac{\mathbf{u}}{(\mathbf{u},\mathbf{v})})\:\; \mathbf{u}\in G(I))$$
and
$$I\cap J=([\mathbf{u},\mathbf{v}]\:\; \mathbf{u}\in G(I), \mathbf{v}\in G(J)).$$
In particular, both $I:\mathbf{v}$ and $I\cap J$ are monomial ideals.
\subsection{Simplicial complexes}
A {\it simplicial complex} $\Delta$ on $[n]$ is a collection of subsets of $[n]$ such that if $F_1\subseteq F_2\subseteq [n]$ and $F_2\in \Delta$, then $F_1\in \Delta$. Each element of $\Delta$ is called a {\it face} of $\Delta$. The {\it dimension} of a face $F$ is $\dim(F)=|F|-1$ and the dimension of $\Delta$ is defined to be the number $\max\{\dim(F)\:\; F\in \Delta\}$. A {\it facet} is a maximal face of $\Delta$ with respect to inclusion. We call $\Delta$ to be {\it pure} if each facet of $\Delta$ has the  same dimension.  Let $\mathcal{F}(\Delta)$ denote the set of facets of $\Delta$. When $\mathcal{F}(\Delta)=\{ F_1,\ldots,F_m\}$, we write $\Delta=\langle F_1,F_2,\ldots,F_m\rangle$.
A simplicial complex is called {\it shellable} (see \cite{BW97}) if its facets can be ordered $F_1,\ldots,F_m$ such that for all $2\leq j\leq m$ the subcomplex $$\langle F_1,\ldots,F_{j-1}\rangle\cap \langle F_j\rangle$$ is pure of dimension  $\dim(F_j)-1$.

For $F\in\Delta$, we set $\mathbf{x}_F=\prod_{i\in F}x_i$. The {\it Stanley-Reisner ideal} $I_{\Delta}$ is defined to be the ideal $$I_{\Delta}=(\mathbf{x}_{F}\:\ F\notin \Delta)$$ and the facet ideal of $\Delta$ is defined to be the ideal $$I(\Delta)=(\mathbf{x}_F\:\; F\in \mathcal{F}(\Delta)).$$
The Stanley-Reisner ring of $\Delta$ is the ring $\mathbb{K}[\Delta]=\mathbb{K}[x_1,x_2,\ldots,x_n]/I_{\Delta}$. It is known that the Krull dimension of $\mathbb{K}[\Delta]$ is $\dim(\Delta)+1$.
\subsection{Alexander Duality} Given a simplicial complex $\Delta$ on $[n]$, the {\it Alexander dual} of $\Delta$  is defined to be the complex $$\Delta^{\vee}=\{[n]\setminus F\:\; F\notin \Delta\}.$$
It was  proved in \cite{EG}  that $I_{\Delta}$  has a linear resolution if and only if $\mathbb{K}[\Delta^{\vee}]$ is Cohen-Macaulay. Lately it was proved in \cite{HRW99} that  $I_{\Delta}$  is componentwise linear if and only if $\mathbb{K}[\Delta^{\vee}]$ is sequentially Cohen-Macaulay.

\section{Linear resolutions and quasi-linearity}

In this section we introduce the concept of quasi-linearity (see Definition~\ref{weakly}) and prove that if a monomial ideal has a linear resolution then it is quasi-linear.
Some consequences of this result are presented and two classes of monomial ideals which are quasi-linear are given.

 Recall some notions from \cite{DDL}. Let $\Delta$ be   a simplicial complex. If $F$ is a facet of $\Delta$, then $\Delta_F$  denotes the simplicial complex whose facet set is $\mathcal{F}(\Delta)\setminus \{F\}$.  Following \cite{DDL}, we say  $\Delta_F$ to $\Delta$ is  a {\it shelling move} if $\langle F\rangle \cap \Delta_F$ is pure of dimension $\dim(F)-1$. If $\Gamma$ is a subcomplex of $\Delta$ generated by facets of $\Delta$, we say $\Delta$ is {\it shelled over } $\Gamma$ if there exists a sequence of shelling moves which take $\Gamma$ to $\Delta$.

By  \cite[Lemma 1.5.3]{HH}, if $\Delta=\langle F_1,F_2,\ldots, F_k\rangle$, then $I_{\Delta^{\vee}}=I(\overline{\Delta})=(\mathbf{x}_{\overline{F_1}}, \mathbf{x}_{\overline{F_2}},\ldots,\mathbf{x}_{\overline{F_k}})$, where $\overline{F}=[n]\setminus F$ for a subset $F\subseteq [n]$.

\begin{Lemma} \label{F} Let $\Delta=\langle F_1,F_2,\ldots,F_k\rangle$. The following are equivalent:

\begin{enumerate}
  \item $\Delta_{F_k}$ to $\Delta$ is a shelling move;
  \item $I_{\Delta_{F_k}^{\vee}}:\mathbf{x}_{\overline{F_k}}$ is generated by variables.
\end{enumerate}
\end{Lemma}

\begin{proof} Since $I_{\Delta_{F_k}^{\vee}}=(\mathbf{x}_{\overline{F_1}}, \mathbf{x}_{\overline{F_2}},\ldots, \mathbf{x}_{\overline{F_{k-1}}})$, it is easy to see that  statement (2) is equivalent to having $\ell\in \overline{F_i}\setminus \overline{F_k}$   for any $1\leq i\leq k-1$ such that $\{\ell\}=\overline{F_j}\setminus \overline{F_k}$ for some  $1\leq j\leq k-1$, and this is the case if and only if  for any $1\leq i\leq k-1$, there exists $j\in  [k-1]$ such that $|F_j\cap F_k|=|F_k|-1$ and $F_i\cap F_k\subseteq F_j\cap F_k$.  Clearly, the latter is equivalent to statement (1).
\end{proof}

\begin{Proposition} \label{squarefree}  Let $I$ be a squarefree monomial ideal with $G(I)=\{\mathbf{u}_1,\ldots, \mathbf{u}_k\}$. Denote by $I_i$ the ideal $(\mathbf{u}_1,\ldots, \mathbf{u}_{i-1}, \mathbf{u}_{i+1},\ldots, \mathbf{u}_k)$.  If $I$ has a $d$-linear resolution, then for $i=1,\ldots,k$,
one has:
\begin{enumerate}
  \item $I_i:\mathbf{u}_i$ is generated by variables;
  \item  $\reg(I_i)\leq d+1$.
\end{enumerate}
\end{Proposition}

\begin{proof} (1) Since $I$ has a $d$-linear resolution, $\deg(\mathbf{u}_i)=d$ for $i=1,\ldots,k$. Let $\Delta$ be the simplicial complex with $I=I_{\Delta}$. Then $\Delta^{\vee}$ is a Cohen-Macaulay complex over $\mathbb{K}$ by Alexander duality (see e.g. \cite[Theorem 8.1.9]{HH}), and $\mathcal{F}(\Delta^{\vee})=\{F_1,\ldots, F_k\}$ with $\mathbf{u}_i=\mathbf{x}_{\overline{F_i}}$ for $i=1,\ldots,k$. By \cite[Lemma 1.1]{DDL}, $(\Delta^{\vee})_{F_i}$ to $\Delta^{\vee}$ is a shelling move. Note that $I_i=I_{(\Delta^{\vee})_{F_i}}$, it follows from Lemma~\ref{F} that $I_i:\mathbf{u}_i$ is generated by variables.

(2) Define a map $\pi: R[-d]\rightarrow \frac{I}{I_i}$ by $\pi(f)=f\mathbf{u}_i+I_i$ for all $f\in R$. Then $\pi$ is a homogeneous epimorphism of graded modules with  $\mbox{Ker}\pi=I_i:u_i$. Hence we have the following exact sequence of graded modules \begin{equation}\label{exact}\tag{$\dag$} 0\rightarrow I_i\rightarrow I\rightarrow \frac{R}{I_i:\mathbf{u}_i}[-d]\rightarrow 0.\end{equation} Since $I_i:\mathbf{u}_i$ is generated by variables, the Koszul complex of  variables which generate $I_i:\mathbf{u}_i$ is a minimal free resolution of $\frac{R}{I_i:\mathbf{u}_i}$ and so $\reg(\frac{R}{I_i:\mathbf{u}_i})=0$. It follows that $\reg(\frac{R}{I_i:\mathbf{u}_i}[-d])=\reg(\frac{R}{I_i:\mathbf{u}_i})+d=d$.

Consider the following long exact sequence induced by the short exact sequence (\ref{exact})
$$\cdots\rightarrow \mbox{Tor}_{s+1}^R(\frac{R}{I_i:\mathbf{u}_i}[-d],\mathbb{K})_t\rightarrow \mbox{Tor}_{s}^R(I_i,\mathbb{K})_t\rightarrow  \mbox{Tor}_{s}^R(I,\mathbb{K})_t\rightarrow\cdots. $$
Thus, if $t-s>d+1$, then $\mbox{Tor}_{s+1}^R(\frac{R}{I_i:\mathbf{u}_i}[-d],\mathbb{K})_t=\mbox{Tor}_{s}^R(I,\mathbb{K})_t=0$, and this implies $\mbox{Tor}_{s}^R(I_i,\mathbb{K})_t=0$. Hence $\reg(I_i)\leq d+1$.
\end{proof}

Using the tool of polarization, we may extend the above result  to the case of arbitrary  monomial ideals. We refer \cite[Section 1.6]{HH} as a good  introduction to the polarization.

\begin{Theorem}\label{main} Let $I$ be a  monomial ideal with $G(I)=\{\mathbf{u}_1,\ldots, \mathbf{u}_k\}$. Denote by $I_i$ the ideal $(\mathbf{u}_1,\ldots, \mathbf{u}_{i-1}, \mathbf{u}_{i+1},\ldots, \mathbf{u}_k)$.  If $I$ has a $d$-linear resolution, then for all $i=1,\ldots,k$, one has

\begin{enumerate}
\item  $\reg(I_i)\leq d+1$;

  \item $I_i:\mathbf{u}_i$ is generated by variables.
  \end{enumerate}
  \end{Theorem}

  \begin{proof} Let $J\subseteq T:=\mathbb{K}[x_{1,1},\ldots,x_{1,a_1}, \ldots, x_{n,1},\ldots, x_{n,a_n}]$ be the polarization of $I$ with $G(J)=\{\mathbf{v}_1,\ldots,\mathbf{v}_k\}$, where $\mathbf{v}_i$ is the polarization of $\mathbf{u}_i$ for $i=1,\ldots,k$.  By \cite[Corollary 1.6.3]{HH}, $\reg(J)=\reg(I)$ and so $J$ is a squarefree monomial ideal with a $d$-linear resolution. This implies that $\reg(J_i)\leq d+1$ by Proposition~\ref{squarefree}, where $J_i:=(\mathbf{v}_1,\cdots,\mathbf{v}_{i-1},\mathbf{v}_{i+1},\ldots,\mathbf{v}_k)$. From this and since  $J_i$ is the polarization of $I_i$, it follows that $\reg(I_i)\leq d+1$ by Proposition~\ref{squarefree}. This proves (1).

  For the proof of (2), we note that the short exact sequence (\ref{exact}) yields the following long exact sequence:  $$\cdots\rightarrow \mbox{Tor}_{s}^R(I,\mathbb{K})_t\rightarrow \mbox{Tor}_{s}^R(\frac{R}{I_i:\mathbf{u}_i}[-d],\mathbb{K})_t\rightarrow \mbox{Tor}_{s-1}^R(I_i,\mathbb{K})_t\rightarrow  \cdots. $$
  By (1), we have $\mbox{Tor}_{s-1}^R(I_i,\mathbb{K})_t=\mbox{Tor}_{s}^R(I,\mathbb{K})_t=0$ whenever $t-s>d$ and so $\mbox{Tor}_{s}^R(\frac{R}{I_i:\mathbf{u}_i}[-d],\mathbb{K})_t=0$ for $t-s>d$. This implies $\reg(I_i:\mathbf{u}_i)=1$ and thus $I_i:\mathbf{u}_i$ is generated by variables.
  \end{proof}

\begin{Example} \em Let $I$ be a monomial ideal with exactly two generators. Then it follows from Theorem~\ref{main} that $I$ has a linear resolution if and only if $I$ has linear quotients with respect to any order. For example, the ideal $(x^2,y^2)$ does not have  a linear resolution since $x^2:y^2=x^2$. On the other hand, $(x^2,xy)$ has a linear resolution as $x^2:xy=x$.  More generally, a monomial ideal $I$ of $R$ with exactly two generators has a linear resolution  if and only if there exist a monomial $\mathbf{v}$ and $i\neq j\in [n]$ such that  $I=(\mathbf{v}x_i,\mathbf{v}x_j).$
  \end{Example}

 Theorem~\ref{main} has several interesting consequences. First we  characterize monomial ideals with linear quotients by the existence   of  a sequence of monomial ideals with a linear resolution.

\begin{Corollary} \label{inter} Let $I$ be a   monomial ideal generated in a single degree. Then the following are equivalent:

  \begin{enumerate}
    \item $I$ has linear quotients;
    \item There is an ordering $\mathbf{u}_1,\ldots, \mathbf{u}_k$ of $G(I)$ such that $(\mathbf{u}_1,\mathbf{u}_2,\ldots,\mathbf{u}_i)$ has  a linear resolution for each $i=1,2,\ldots,k$.
  \end{enumerate}
  \end{Corollary}
  \begin{proof} $(1)\Rightarrow (2)$ is well-known, see e.g. the proof of \cite[Proposition 8.2.1]{HH} and  $(2)\Rightarrow (1)$ follows directly from Theorem~\ref{main}.
  \end{proof}

To state other consequences, we need some more notation and definitions.

\begin{Notation}\em  If $I$ is a monomial ideal with $\mathbf{u}\in G(I)$, we use $I\setminus_{\mathbf{u}}$ to denote the monomial ideal  generated by monomials in $G(I)\setminus \{\mathbf{u}\}$.
\end{Notation}

\begin{Definition}\em  \label{critical} We call a monomial ideal $I$ to be a {\it critical linear ideal} if $I$ has a linear resolution, but $I\setminus_{\mathbf{u}}$  has not a linear resolution for each $\mathbf{u}\in G(I)$. By convention, the zero ideal is a critical linear ideal.
\end{Definition}

\begin{Definition} \em  \label{linear}
 Let $J\subseteq I$ be monomial ideal.  We say $I$ is  {\it linear over} $J$ if there exist $r>0$ and monomials $\mathbf{u}_1,\ldots,\mathbf{u}_r$ such that $G(I)=G(J)\cup \{ \mathbf{u}_1,\ldots,\mathbf{u}_r\}$ and $(J,\mathbf{u}_1,\ldots,\mathbf{u}_{i-1}):\mathbf{u}_i$ is generated by linear forms for $i=1,\ldots,r$.
\end{Definition}

\begin{Corollary} \label{minimal} Let $I$ be a monomial ideal generated in a single degree. Then $I$  has a linear resolution if and only if there is a critical linear ideal $J$ such that $I$ is linear over  $J$.
\end{Corollary}

\begin{proof} The necessity follows from Theorem~\ref{main} and the sufficiency  follows either  by  Proposition~\ref{simple} or by the following observation:   if $I_i$  has a linear resolution and $I_i:\mathbf{u}$ is generated by variables in the short exact sequence  (\ref{exact}), then $I$ has a linear resolution.
\end{proof}

The result above suggests that the study of monomial ideals with linear resolutions boils down  to the study of critical linear monomial ideals, which has the following property.
\begin{Corollary} Let $I$ be a monomial ideal generated in degree $d$. Then the following are equivalent:

\begin{enumerate}
  \item $I$ is a critical linear ideal;
  \item $\reg(I)=d$  and $\reg(I\setminus_{\mathbf{u}})=d+1$ for all $\mathbf{u}\in G(I)$.
\end{enumerate}

\end{Corollary}

\begin{Example} \em Let $I$ be a monomial ideal  with $|G(I)|=3$. If $I$ has a linear resolution, but has not linear quotients, then $I$ is a critical linear ideal.
\end{Example}

We are led to the following definition by Theorem~\ref{main}.

  \begin{Definition} \label{weakly}\em We call a monomial ideal $I$ to be {\it quasi-linear } if  the colon ideal  $I\setminus_{\mathbf{u}}: \mathbf{u}$  is generated by linear forms (i.e., variables) for every $\mathbf{u}\in G(I)$.
  \end{Definition}

  By Theorem~\ref{main}, if a monomial ideal has a linear resolution then it is quasi-linear. But the converse is not true since whether a monomial ideal has a linear resolution depends on the characteristic of the base field $\mathbb{K}$. We may see this more clear by characterizing all the quadratic monomial ideals which are quasi-linear .

We refer basic notions and facts about graphs we need to \cite[Chapter 9]{HH}. Let $I$ be a monomial ideal of $R$ generated in  degree 2. We associate to it a simple graph $G_I$, which is defined as follows: $V(G_I)=[n]\cup \{\underline{i}\:\; x_i^2\in I\}$ and $E(G_I)=\{\{i,j\}\:\; x_ix_j\in I \mbox{ and } i\neq j\}\cup \{\{i,\underline{i}\}\:\; x_i^2\in I\}$.  Recall for a simple graph $G$, a subset $M\subseteq E(G)$ is called a {\it matching} of $G$ if $e\cap e'=\emptyset$ for any distinct $e,e'$ belonging to $M$, and a matching is called an {\it induced matching} if these pairwise disjoint edges form an induced subgraph of $G$.  The induced matching number of $G$, denoted by  $\mathrm{indmat}(G)$, is the maximum  cardinality of induced matchings of $G$.

\begin{Proposition} \label{quadratic} Let  $I$ be a monomial ideal of $R$ generated in degree 2. Then $I$ is quasi-linear  if and only if $\mathrm{indmat}(G_I)=1$.
\end{Proposition}

\begin{proof} If $\mathrm{indmat}(G_I)\geq 2$, there is an induced subgraph of $G_I$ consisting of exactly two disjoint  edges $e_1$ and $e_2$. By symmetry, there are actually three cases to consider:

(1)  $e_1=\{1,2\}$ and $e_2=\{3,4\}$. We claim $x_3x_4$ is a minimal monomial generator of $I\setminus_{x_1x_2}:x_1x_2$. First we note that $x_3x_4\in I\setminus_{x_1x_2}:x_1x_2$. If it is not a minimal monomial generator, we have either $x_3$ or $x_4$ belongs to $I\setminus_{x_1x_2}:x_1x_2$, and so at least one of monomials $x_1x_3,x_1x_4,x_2x_3,x_2x_4$ belongs to $I$. This is impossible since $e_1,e_3$ form an induced graph of $G_I$ and so we proves the claim. From this it follows that $I$ is not quasi-linear.

(2) $e_1=\{1,\underline{1}\}$ and $e_2=\{2,\underline{2}\}$.  As in the first case, we can prove that   $x_1^2$ is a minimal monomial generator of $I\setminus_{x_2^2}:x_2^2$ and so $I$ is not quasi-linear.

(3)   $e_1=\{1,2\}$ and $e_2=\{3,\underline{3}\}$. In this case we have  $x_3^2$ is a minimal monomial generator of $I\setminus_{x_1x_2}:x_1x_2$ and so $I$ is not quasi-linear.

Thus, $I$ is not quasi-linear since all the possible cases are considered.

Conversely, suppose that  $I$ is not quasi-linear. Then there is a monomial in $G(I)$, say $\mathbf{u}$, such that $I\setminus_{\mathbf{u}}:\mathbf{u}$ has a minimal monomial generator of degree 2, say $\mathbf{v}$.  It is clear that $\mathbf{v}\in I$.  By symmetry, we only need to consider the following four cases.

 (1) $\mathbf{u}=x_3^2$ and $\mathbf{v}=x_1x_2$. We claim the edges $\{1,2\}$ and $\{3,\underline{3}\}$ form an induced subgraph of $G_I$. In fact, if this is not the case, then either $\{1,3\}$ or $\{2,3\}$ are edges of $G_I$. This means either $x_1x_3$ or $x_2x_3$ belongs to $I$ and so we have $x_1$ or $x_2$ belongs to $I\setminus_{\mathbf{u}}:\mathbf{u}$. This is a contradiction and so the claim has been proved. From this it follows that  $\mathrm{indmatch}(G_I)\geq 2$.

We   now  consider   the other cases: (2) $\mathbf{u}=x_1x_2$ and $\mathbf{v}=x_3x_4$, (3)  $\mathbf{u}=x_1x_2$ and   $\mathbf{v}=x_3^2$ and (4) $\mathbf{u}=x_1^2$ and $\mathbf{v}=x_2^2$.  In a similar way as in (1)  we obtain that $\mathrm{indmatch}(G_I)\geq 2$. Thus  the proof is complete.
\end{proof}
By the definition of $G_I$, it is easy to see that $G_I$ is isomorphic to the graph $G_{I^{\mathrm{p}}}$, where $I^{\mathrm{p}}$ denotes the polarization of $I$. Thus the following is immediate.

\begin{Corollary} \label{P}Let  $I$ be a monomial ideal of $R$ generated in degree 2. Then $I$ is quasi-linear  if and only if its polarization $I^{\mathrm{p}}$ is quasi-linear .
\end{Corollary}

We do not know if Corollary~\ref{P} is true for monomial ideals generated in degree more than two. The squarefree version of Corollary~\ref{quadratic} is as follows.

\begin{Corollary} \label{chordal} Let $G$ be  a simple graph. Then $I(G)$ is quasi-linear  if and only if $\overline{G}$ contains no induced 4-cycles.
\end{Corollary}

\begin{proof} It follows from  the fact that $\mathrm{indmatch}(G)\geq 2$ if and only if $\overline{G}$ contained an induced 4-cycle.
\end{proof}

  Let $\mathfrak{m}$ be the maximal monomial  ideal  $(x_1,\ldots,x_n)$ of $R$, and $d$ an integer $\geq 2$. Then $\mathfrak{m}^d$ has a linear resolution since it has linear quotients with respect to the lex order.  We now consider  when $\mathfrak{m}^d\setminus_{\mathbf{u}}$  is quasi-linear, where $\mathbf{u}$ is a monomial of degree $d$.
We first deal with the case when $n=2$.

\begin{Lemma} \label{n=2} If $n=2$, we write $\mathfrak{m}=(x,y)$. Then $\mathfrak{m}^d\setminus_{\mathbf{u}}$ is not quasi-linear  for any monomial $\mathbf{u}=x^ay^b$ with $a>0,b>0$ and $a+b=d$.
\end{Lemma}
\begin{proof} Set $I=\mathfrak{m}^d\setminus_{\mathbf{u}}$ and let $\mathbf{v}=x^{a-1}y^{b+1}$.   Then $x^{d-a+1}=x^{b+1}\in I\setminus_{\mathbf{v}}:\mathbf{v}$ but $x\notin I\setminus_{\mathbf{v}}:\mathbf{v}$. This implies $I\setminus_{\mathbf{v}}:\mathbf{v}$ is not generated by variables and so $I$ is not quasi-linear .
\end{proof}

\begin{Proposition} Let $I=\mathfrak{m}^d$ and $\mathbf{u}\in G(I)$. Then $I\setminus_{\mathbf{u}}$ is quasi-linear  if and only if $|\mathrm{supp}(\mathbf{u})|\neq 2$.
\end{Proposition}

\begin{proof} Set $J=I\setminus_{\mathbf{u}}$. Suppose that $|\mbox{supp}(\mathbf{u})|= 2$. We may assume that $\mbox{supp}(\mathbf{u})=\{1,2\}$. Let $A$  denote the subset of $G(J)$ consisting of monomials  whose supports are  included in $\{1,2\}$. By Lemma~\ref{n=2}, there is a monomial $\mathbf{v} \in A$  such that the colon ideal $(A)\setminus_{\mathbf{v}}:\mathbf{v}$ has a minimal monomial generator, say $\overline{\mathbf{v}}$,  of degree more than 1. Note that $\mbox{supp}(\overline{\mathbf{v}})\subseteq \{1,2\}$. We claim that $\overline{\mathbf{v}}$ is also a minimal monomial generator of $J\setminus_{\mathbf{v}}: \mathbf{v}$. In fact, for any $\mathbf{w}\in G(J)\setminus A$, we have $\frac{[\mathbf{w},\mathbf{v}]}{\mathbf{v}}$ can not divide $\overline{\mathbf{v}}$ since $\mathrm{supp}(\frac{[\mathbf{w},\mathbf{v}]}{\mathbf{v}})\cap \{3,\ldots,n\}\neq \emptyset$. This implies $\overline{\mathbf{v}}\in G(J\setminus_{\mathbf{v}}: \mathbf{v})$, as claimed. From this it follows that $J$ is not quasi-linear .

If $|\mbox{supp}(\mathbf{u})|= 1$, then $I\setminus_{\mathbf{u}}$ has a linear resolution, see e.g. Example~\ref{one}. In particular, it is quasi-linear. Suppose now that $|\mbox{supp}(\mathbf{u})|\geq 3$.  Let $\mathbf{v}\in G(J)=G(I\setminus_{\mathbf{u}})$. If $|\mbox{supp}(\mathbf{v})|=1$, we may assume $\mathbf{v}=x_1^d$. We claim that $J\setminus_{\mathbf{v}}:\mathbf{v}=(x_2,\ldots,x_n)$.  In fact, for any $2\leq i\leq n$, we have $\frac{x_i\mathbf{v}}{x_1}=x_ix_1^{d-1}\notin \{\mathbf{u}, \mathbf{v}\}$ since $|\supp(\frac{x_i\mathbf{v}}{x_1})|=2$. This implies $x_i\mathbf{v}\in J\setminus_{\mathbf{v}}$ and so $x_i\in J\setminus_{\mathbf{v}}:\mathbf{v}$. It is clear that $x_1\notin J\setminus_{\mathbf{v}}:\mathbf{v}$ and so we prove the claim. If $|\mbox{supp}(\mathbf{v})|\geq 3$, we show that $J\setminus_{\mathbf{v}}:\mathbf{v}=\mathfrak{m}$. First, we may assume $\{1,2,3\}\subseteq \mbox{supp}(\mathbf{v})$. Next, for any $i\in [n]$, note that
there is at least one of monomials in  $\{\frac{x_i\mathbf{v}}{x_1}, \frac{x_i\mathbf{v}}{x_2},\frac{x_i\mathbf{v}}{x_3}\}$ does not belong to $\{\mathbf{u},\mathbf{v}\}$. This implies $x_i\mathbf{v}\in J\setminus_{\mathbf{v}}$ and so $x_i\in J\setminus_{\mathbf{v}}:\mathbf{v}$. Thus, we have $J\setminus_{\mathbf{v}}:\mathbf{v}=\mathfrak{m}$ indeed. From those it follows that $J$ is quasi-linear .
\end{proof}

 \section{Linearity and Strong Linearity}

 Let $I$ and $J$ be  monomial ideals. Recall that the notion that  $I$ is linear over $J$ has been defined in Definition~\ref{linear}.  In this section we first prove if $I$ is linear over $J$,  then  $I$  has a linear resolution whenever $J$ has a linear resolution, but not vice versa. To make up for this defect, we introduce the notion of strong linearity and  prove  if $I$   is strongly linear over $J$, then $I$ is quasi-linear if and only if $J$ is quasi-linear, and  $I$ has a linear resolution if and only if  $J$ has a linear resolution.

  \begin{Proposition} \label{simple} Let $J\subset I$ be  monomial ideals such that $I$ is linear over $J$.
  \begin{enumerate}
   \item If $J$ is componentwise linear, then so is $I$;
    \item If $J$ has a linear resolution and $I$ is generated in a single degree, then  $I$ has a linear resolution.

  \end{enumerate}
    \end{Proposition}

    \begin{proof} It was shown in the proof of \cite[Theorem 8.2.5]{HH} that if $f_1,\ldots, f_m$ is a minimal homogeneous system of generators of a graded ideal such that $(f_1,\ldots,f_{m-1})$ is componentwise linear and $(f_1,\ldots,f_{m-1}):f_m$ is generated by linear forms, then this ideal is componentwise linear.
    The statement (1) follows directly from this fact.

    Since a componentwise linear ideal generated in a single degree has a linear resolution, the statement (2) is a direct sequence of (1).
    \end{proof}
    Proposition~\ref{simple} has the following application in the area of simplicial complexes.
  \begin{Proposition} Suppose that $\Delta$ is shelled over $\Gamma$.
\begin{enumerate}
     \item If $\Gamma$ is a Cohen-Macaulay complex over $\mathbb{K}$, then so is $\Delta$;
      \item If $\Gamma$ is a sequentially Cohen-Macaulay complex over $\mathbb{K}$, then so is $\Delta$.  \end{enumerate}
 \end{Proposition}
\begin{proof}  In view of Lemma~\ref{F}, this is  the Alexander dual of Proposition~\ref{simple}.
\end{proof}

The converse of Proposition~\ref{simple}.(2) is not true.  In other words, when  $I$ is linear over $J$, we cannot deduce that $J$ has a linear resolution from the condition that $I$ has a linear resolution. We now introduce a new notion, which is stronger  than ``{\it linear over}", and it will be the main topic of this section.

 \begin{Lemma} \label{1} Let $I$ be a monomial ideal generated in  degree $d$, and $\mathbf{u}$ a monomial of degree $d-1$. Then the following are equivalent:
 \begin{enumerate}
   \item $I:\mathbf{u}$ is generated by variables;
   \item  There are a subset $A$ of  $[n]$ such that $\mathbf{u}x_i\in G(I)$ for each $i\in A$ and $\mathrm{supp}(\frac{[\mathbf{v},\mathbf{u}]}{\mathbf {u}})\cap A\neq \emptyset$ for each $\mathbf{v}\in G(I)$;
       \item  There are  a subset $A$ of  $[n]$ such that $\mathbf{u}x_i\in G(I)$ for all $i\in A$ and $[\mathbf{u},\mathbf{v}]\in \mathbf{u}(x_i\:\; i\in A)$ for each $\mathbf{v}\in G(I)$.
 \end{enumerate}
 \end{Lemma}

 \begin{proof} The proof is easy and so we omit it.
 \end{proof}

 \begin{Definition} \em  Let $I$ be  a monomial ideal generated in degree $d$, and $\mathbf{u}$ a monomial of degree $(d-1)$.  We say that  the monomial $\mathbf{u}$ is {\it   strongly linear} over $I$ if one of the equivalence conditions in Lemma~\ref{1} is satisfied.
 \end{Definition}

 \begin{Proposition} \label{useful} Let $I$ be  a monomial ideal generated in  degree $d$, and $\mathbf{u}$ a monomial strongly linear over $I$. Then the following statements hold true. \begin{enumerate}

    \item  For any $i\in [n]$ with $x_i\mathbf{u}\notin I$, $x_i\mathbf{u}$ is linear over $I$;

                  \item If  $I$ is a squarefree monomial ideal, then $\mathbf{u}$ is squarefree;

\item   $\mathbf{u}$ is also strongly linear over $I+x_i\mathbf{u}$ for any $i\in [n]$.
                \end{enumerate}
 \end{Proposition}
 \begin{proof} (1) We may assume that $I:\mathbf{u}=(x_j\:\; j\in A)$, where $A$ is a subset of $[n]$.  Note that $i\notin A$ since $x_i\mathbf{u}\notin I$. We claim that $I:\mathbf{u}x_i=(x_j\:\; j\in A)$.  Since $I:\mathbf{u}\subseteq I:\mathbf{u}x_i$, one has $(x_j\:\; j\in A)\subseteq I:\mathbf{u}x_i$. Conversely, for any $\mathbf{v}\in G(I)$, note that $\frac{[\mathbf{v}, \mathbf{u}x_i]}{\mathbf{u}x_i}$ is either
$\frac{[\mathbf{v}, \mathbf{u}]}{\mathbf{u}}$ or $(\frac{[\mathbf{v}, \mathbf{u}]}{\mathbf{u}})/x_i$, it follows that $\frac{[\mathbf{v}, \mathbf{u}x_i]}{\mathbf{u}x_i}\in (x_j\:\; j\in A)$ by Lemma~\ref{1}.(2) together with the known fact $i\notin A$. Hence $ I:\mathbf{u}x_i=(x_j\:\; j\in A)$, as claimed.

 (2) Assume on the contrary that $\mathbf{u}$ is not squarefree. Then $\mathbf{v}:=\prod_{i\in \mathrm{supp}(\mathbf{u})}x_i$ is  a squarefree monomial of degree  $\leq (d-2)$. Since $I$ is squarefree, we have $I:\mathbf{u}=I:\mathbf{v}$ and so  $I:\mathbf{u}$ is generated by monomials in degree at least 2, a contradiction.

(3) Straightforward.
 \end{proof}

\begin{Definition}\em  Let $J\subseteq I$ be monomial ideals generated in degree $d$. We say $I$ is {\it 1-step  strongly linear over} $J$ if there is a monomial  $\mathbf{u}$ which is strongly linear over $J$ such that $I=J+(\mathbf{u}F)$ for some subset $F$  of $\{x_1,\ldots,x_n\}$. For a positive integer $s$, that $I$ is {\it $s$-step strongly  linear over $J$} is defined by recursion. We say  $I$ is {\it strongly linear over} $J$  if $I$ is $s$-step strongly linear over $J$ for some $s>0$.
 \end{Definition}

\begin{Example}\em  Let $G$ be a simple graph on $[n]$ and fix $i\in [n]$. \begin{enumerate}
                                                                              \item The variable $x_i$ is strongly linear over $I(G)$ if and only if $N_G(i)$ is a cover of $G$, that is, $N_G(i)\cap e\neq \emptyset$ for any $e\in E(G)$. Here, $N_G(i)$ is the open neighborhood $\{j\in [n]\:\; \{i,j\}\in E(G)\}$ of $i$ in $G$.
                                                                              \item Let $x_i$ be strongly linear over $I(G)$ and  $j\in [n]\setminus N_G[i]$. Then $I(G)$ is quasi-linear if and only if $I(G)+(x_ix_j)$ is quasi-linear. Here, $N_G[i]$ is the closed neighborhood $N_G(i)\cup \{i\}$ of $i$ in $G$.
                                                                            \end{enumerate}

\end{Example}
\begin{proof} (1) This follows  from the definitions.

(2) Let $G'$ be the graph associated to the ideal $I+(x_ix_j)$. Then $G$ is a subgraph of $G'$. If $I(G)$ is not quasi-linear, there exist edges $e_1,e_2$, which form an induced matching of $G$ by Proposition~\ref{quadratic}. Note that $i\notin e_1\cup e_2$  by (1) and  $\{e_1,e_2\}$ is yet an induced matching of $G'$, it follows that $I+(x_ix_j)$ is not quasi-linear.

Conversely, suppose that  $I+(x_ix_j)$ is not quasi-linear.  We may assume $\{e_1,e_2\}$ is an induced matching of $G'$, see Proposition~\ref{quadratic}. Since  $i\notin e_1\cup e_2$,  we have   $\{e_1,e_2\}\subseteq E(G)$, and so it is also an induced matching of $G$. From this it follows that $I(G)$ is not quasi-linear. \end{proof}

More generally, we have the following proposition.

\begin{Proposition} \label{quasikeep} Let $I$ be a monomial ideal generated in degree $d$ and $\mathbf{u}$ a   monomial strongly linear over $I$. Let $k\in [n]$.  Then the following statements are equivalent:
\begin{enumerate}
  \item $I$ is not quasi-linear;
  \item  $I+(\mathbf{u}x_k)$ is not quasi-linear.
\end{enumerate}

\end{Proposition}

\begin{proof} By Lemma~\ref{1}, we may write $G(I)=\{\mathbf{u}x_1,\ldots,\mathbf{u}x_r,\mathbf{u}_1,\ldots,\mathbf{u}_p\}$ such that  $\mathbf{u}$ does not divide $\mathbf{u}_i$ for each $i=1,\ldots p$. In addition we may assume $k=r+1$.

If $I$ is not quasi-linear, then  by  Proposition~\ref{useful}.(1),  we may  assume that $$(\mathbf{u}x_1,\ldots,\mathbf{u}x_r,\mathbf{u}_1,\ldots,\mathbf{u}_{p-1}):\mathbf{u}_p,$$ which we denote by $J$, is not generated by variables.  We claim that $$(\mathbf{u}x_1,\ldots,\mathbf{u}x_r,\mathbf{u}x_{r+1}, \mathbf{u}_1,\ldots,\mathbf{u}_{p-1}):\mathbf{u}_p$$ is not generated by variables either. In fact, if this is not the case, then $\frac{[\mathbf{u}x_{r+1},\mathbf{u}_p]}{\mathbf{u}_p}$ must be a variable, say $x_i$, and moreover, this variable (i.e., $x_i$)  divides all minimal monomial generators of degree at least 2 of  $J$. Since $\mathbf{u}$ does not divide $\mathbf{u}_p$, we have $i\neq r+1$.  This implies $\mathbf{u}=\mathbf{v}x_i$ for some monomial $\mathbf{v}$ and $\mathbf{u}_p=x_{r+1}\mathbf{v}x_j$ for some $j\neq i$.  By Lemma~\ref{1}, $\frac{[\mathbf{u}_p, \mathbf{u}]}{\mathbf{u}}=x_jx_{r+1}\in (x_1,\ldots,x_r)$,  so  $j\in \{1,\ldots,r\}$.  From this, it follows that $x_i=\frac{[\mathbf{u}x_j,\mathbf{u}_p]}{\mathbf{u}_p}\in J$, a contradiction. This proves the claim and $I+\mathbf{u}(x_{r+1})$ is not quasi-linear.

Suppose now $I$ is quasi-linear. In view of  Proposition~\ref{useful}.(1), it is enough to show $$(\mathbf{u}x_1,\ldots,\mathbf{u}x_r,\mathbf{u}x_{r+1}, \mathbf{u}_1,\ldots,\mathbf{u}_{p-1}):\mathbf{u}_p$$ is generated by variables. Let $\mathbf{v}$ denote the greatest common divisor of $\mathbf{u}$ and $\mathbf{u}_p$. Then we may  write $\mathbf{u}=\mathbf{vw}$ and $\mathbf{u}_p=\mathbf{vw}_p$, where $\mathbf{w}$ and $\mathbf{w}_p$ are monomials. By Lemma~\ref{1}.(2), $\mathbf{w}_p=\frac{[\mathbf{u},\mathbf{u}_p]}{\mathbf{u}}$ is divided by at least one of variables $x_1,\ldots,x_r$, say $x_1$. From this it follows that $\frac{[\mathbf{u}x_1,\mathbf{u}_p]}{\mathbf{u}_p}=\frac{\mathbf{u}x_1}{(\mathbf{u}x_1,\mathbf{u}_p)}=\frac{\mathbf{u}x_1}{\mathbf{v}x_1}=\mathbf{w}$. Here, $(\mathbf{u}x_1,\mathbf{u}_p)$ denotes the  greatest common divisor of $\mathbf{u}x_1$  and $\mathbf{u}_p$.  Since $\mathbf{w}$ divides $\frac{[\mathbf{u}x_{r+1},\mathbf{u}_p]}{\mathbf{u}_p}$,  we have $$(\mathbf{u}x_1,\ldots,\mathbf{u}x_r,\mathbf{u}x_{r+1}, \mathbf{u}_1,\ldots,\mathbf{u}_{p-1}):\mathbf{u}_p=(\mathbf{u}x_1,\ldots,\mathbf{u}x_r, \mathbf{u}_1,\ldots,\mathbf{u}_{p-1}):\mathbf{u}_p$$ is generated by variables.
\end{proof}
We now come to the first main result of this section.
\begin{Corollary} \label{firstmain} Let $J\subseteq I$ be monomial ideals generated in degree $d$ such that $I$ is strongly linear over $J$. Then $I$ is quasi-linear if and only if $J$ is quasi-linear.
\end{Corollary}

\begin{proof} It follows from Proposition~\ref{quasikeep} as well as Proposition~\ref{useful}.(3).
\end{proof}

 In the  second main result of this section we compare Betti numbers of $I$ and $J$ when $I$ is strongly linear over $J$.
  \begin{Theorem}\label{main1} Let $I$ be a monomial ideal generated in degree $d$ and $\mathbf{u}$ a  monomial strongly linear over $I$. Then for any  subset $A$ of  $[n]$, we have
  \begin{enumerate}
    \item $\beta_{s,t}(I+\mathbf{u}(x_i\:\; i\in A))=\beta_{s,t}(I)$ for all pairs $s,t$ with $t-s\neq d$;
    \item $\reg(I+\mathbf{u}(x_i\:\; i\in A))=\reg(I)$;
    \item  $I$  has  a linear resolution if and only if $I+\mathbf{u}(x_i\:\; i\in A)$ has a linear resolution.
  \end{enumerate}
  \end{Theorem}

 \begin{proof} Let $B\subseteq [n]$ be such that $I:\mathbf{u}=(x_i\:\; i\in B)$ and $L=\mathbf{u}(x_i:i\in A\cup B)$. It is easy to see that $I\cap L=\mathbf{u}(x_i\:\; i\in B)$ and $L$ have a $d$-linear resolution and $I+L=I+\mathbf{u}(x_i\:\; i\in A)$. Now, \cite[Proposition 2.4]{BYZ} yields the required results.
 \end{proof}

 \begin{Corollary} Let $J\subseteq I$ be monomial ideals such that $I$ is strongly linear over $J$. Then $I$ has a linear resolution if and only if $J$ has a linear resolution.
\end{Corollary}

 We  give some examples and a proposition (Proposition~\ref{stable}) to show that  monomials  strongly linear over a monomial ideal occur frequently.

\begin{Example} \label{two} \em Let $I=(x_1x_2x_3, x_3^2x_4,x_3x_4^2)$. Then $x_1x_2$ is strongly linear over $I$. Thus, we have $\beta_{i,j}(I)=\beta_{i,j}(I+(x_1x_2^2))=\beta_{i,j}(I+(x_1x_2^2+x_1^2x_2))=\cdots$ for each pair $i,j$ with $j-i\neq 3$.
 \end{Example}

\begin{Example} \label{one} \em Let $I=\mathfrak{m}^d$. For all  $i\in [n]$, we have $x_i^{d-1}$ is strongly linear over $I\setminus_{x_i^d}$ since $I\setminus_{x_i^d}:x_i^{d-1}=(x_1,\ldots,x_{i-1},x_{i+1},\ldots,x_n)$. This implies $I$ is strongly linear over $I\setminus_{x_i^d}$ and so  $I\setminus_{x_i^d}$ has a linear resolution.
 \end{Example}

We proceed to prove that $\mathfrak{m}^d$ is strongly linear over any stable monomial ideal generated in degree $d$. Recall that a monomial ideal $I$ is {\it stable} if for any monomial $\mathbf{u}\in I$, (or equivalently for any  $\mathbf{u}\in G(I)$), and for any $1\leq i\leq \mathrm{m}(\mathbf{u})$, the monomial $\frac{\mathbf{u}x_i}{x_{\mathrm{m}(\mathbf{u})}}$ belongs to $I$. Here,  $\mathrm{m}(\mathbf{u})$ denotes the number $\max\{i\in [n]\:\; x_i \mbox{ divides } \mathbf{u}\}$ for a monomial $\mathbf{u}$ .

To get the desired result, we need some lammas.
\begin{Lemma} \label{key}
Let $I$ be a stable monomial ideal  generated in degree $d$.

\begin{enumerate}
  \item If $\mathbf{v}$ is a monomial of degree $(d-1)$ such that $x_{i_1}\cdots x_{i_k}\mathbf{v}\in I$ with $i_1\geq i_2\geq \cdots \geq i_k\geq \mathrm{m}(\mathbf{v})$, then $x_{i_k}\in I:\mathbf{v}$.
  \item If  $\mathbf{v}$ is a monomial of degree $(d-1)$ such that $(x_1,\ldots,x_k)\subseteq I:\mathbf{v}$ with some $k\geq \mathrm{m}(\mathbf{v})-1$, then $I:\mathbf{v}=(x_1,x_2,\ldots,x_{\ell})$ for some $\ell\geq k$ and $I+\mathbf{v}\mathfrak{m}$ is also a stable monomial ideal.
\end{enumerate}

\end{Lemma}

\begin{proof} (1) It is clear that $x_{i_k}^k\mathbf{v}\in I$. Since $I$ is generated in degree $d$, there is a positive integer $\ell$ and a monomial $\mathbf{u}$ which divides $\mathbf{\mathbf{v}}$ such that $x_{i_k}^{\ell}\mathbf{u}\in G(I)$. By the stability of $I$ we have  $x_{i_k}\mathbf{v}\in I$, as desired.

(2) First we claim  that $I:\mathbf{v}$ is  generated by variables. In fact, if this is not true, there is a monomial $x_{i_1}\cdots x_{i_j}\in G(I: \mathbf{v})$ with $j\geq 2$. Since $(x_1,\ldots,x_k)\subseteq I:\mathbf{v}$, it follows that $i_l\geq k+1\geq \mathrm{m}(\mathbf{v})$ for $l=1,\ldots,j$.
But this is impossible in view of (1) and thus the claim is proved. From this claim together with the stability of $I$ it follows that $I:\mathbf{v}=(x_1,x_2,\ldots,x_{\ell})$ for some $\ell\geq k$.

Now we prove that $I+\mathbf{v}\mathfrak{m}$ is also stable. Since $\{\mathbf{v}x_{1},\ldots,\mathbf{v}x_{\ell}\}\subseteq I$, we have $G(I+\mathbf{v}\mathfrak{m})=G(I)\cup \{\mathbf{v}x_{\ell+1},\ldots,\mathbf{v}x_n\}$.  On the other side, for any $j=1,\ldots, n-\ell$, since $\mathrm{m}(\mathbf{v}x_{\ell+j})=\ell+j,$ we obtain $$\frac{\mathbf{v}x_{\ell+j}x_i}{\mathrm{m}(\mathbf{v}x_{\ell+j})}=\mathbf{v}x_i\in G(I)$$ for any $i\leq \ell+j$. Hence  $I+\mathbf{v}\mathfrak{m}$ is  stable.
\end{proof}

Let $I$  be  a stable monomial ideal in degree $d$. For convenience, we call a monomial $\mathbf{v}$ to be a {\it special monomial} over $I$ provided that $(x_1,\ldots,x_k)\subseteq I:\mathbf{v}$ for some $k\geq \mathrm{m}(\mathbf{v})-1$ and $I\neq I+\mathbf{v}\mathfrak{m}$. By Lemma~\ref{key}.(2),  if $\mathbf{v}$ is a special monomial over $I$, then it is strongly linear over $I$.  We also denote by $\mathfrak{m}_k^t$ the set consisting of monomials $\mathbf{u}$ in degree $t$ with  $\mathrm{supp}(\mathbf{u})\subseteq \{1,\ldots,k\}$, where  $k=1,\dots,n$ and $t\geq 0$. Both notions are only used in Lemma~\ref{key1} and Proposition~\ref{stable} and their proofs.

\begin{Lemma} \label{key1} Let $0\neq I$ be a stable monomial ideal generated in degree $d$ with $I\subsetneq \mathfrak{m}^d$. Then there is at least one  monomial which is  special   over $I$.
\end{Lemma}

\begin{proof} Suppose that  no  monomials  are special  over $I$.  To get a contradiction, we proceed to prove $\mathfrak{m}_n^d\subseteq I$ by induction on $n$.  Since $I$ is a stable monomial ideal, we have $\mathfrak{m}_1^d=(x_1^d)\subseteq I$. Assume now that $\mathfrak{m}_k^d\subseteq I$ for some $1\leq k\leq d-1$. Then for any  $\mathbf{v}\in \mathfrak{m}_k^{d-1}$,  it is easy to see that $(x_1,\ldots, x_k)\subseteq I:\mathbf{v}$.  By assumption, we have $I= I+\mathbf{v}\mathfrak{m}$ and so $I:\mathbf{v}=\mathfrak{m}$. In particular, it follows that $\mathbf{v}x_{k+1}\in I$ and
thus $\mathfrak{m}_k^{d-1}x_{k+1}\subseteq I$. Similarly, we may infer that $\mathfrak{m}_k^{d-2}x_{k+1}^2\subseteq I$ from that $\mathfrak{m}_k^{d-1}x_{k+1}\subseteq I$, and so on. Proceeding in this way, we have $\mathfrak{m}_{k+1}^d\subseteq I$ and then $\mathfrak{m}_n^d\subseteq I$. This means that $\mathfrak{m}^d= I$,  a contradiction.
\end{proof}

    The following is what we want.
\begin{Proposition} \label{stable} Let $I\neq 0$ be a stable monomial ideal generated in degree $d$. Then $\mathfrak{m}^d$ is strongly linear over $I$.
\end{Proposition}

\begin{proof} We proceed by  induction on the number of $G(\mathfrak{m}^d)\setminus G(I)$. If $|G(\mathfrak{m}^d)\setminus G(I)|=0$, there is nothing to prove.  If $I\subsetneq \mathfrak{m}^d$,  then we may assume $\mathbf{v}$ is a monomial special over  $I$ by  Lemma~\ref{key1}. This implies $I\subsetneq I+\mathbf{v}\mathfrak{m}$  and  $I+\mathbf{v}\mathfrak{m}$.is also a stable monomial ideal.  Consequently, $\mathfrak{m}^d$ is strongly linear over $I+\mathbf{v}\mathfrak{m}$ by induction hypothesis. From this  it follows that $\mathfrak{m}^d$ is also strongly linear over $I$, as desired.
\end{proof}

\begin{Example} \em  Let $I$  be the stable monomial ideal $(x_2^2x_3,x_2^3,x_1x_2^2,x_1^2x_2,x_1^3)\subset \mathbb{K}[x_1,x_2,x_3]$.  Denote $I$ by $I_0$. Then $I_0:x_1x_2=(x_1,x_2)$ and so $x_1x_2$ is strongly linear over $I_0$. This implies that $I_1:=I_0+(x_1x_2x_3)$ is an 1-step strongly linear over $I$. Similarly, since $I_1:x_2x_3=(x_1,x_2)$, we set  $I_2=I_1+(x_2x_3^2)$. Then $I_2:x_1^2=(x_1,x_2)$, and  we set $I_3=I_2+(x_1^2x_3)$. Note that $I_3:x_1x_3=(x_1,x_2)$, we set $I_4=I_3+(x_1x_3^2)$.  Then $I_4:x_3^2=(x_1,x_2)$ and  we set $I_5=I_4+(x_3^3)$. Since $I_5=\mathfrak{m}^3$,  it follows that $\mathfrak{m}^3$ is a 5-step strongly linear over $I$.
 \end{Example}

Finally we consider the multi-graded Betti numbers in the linear strands of $J$ and $I$ if $J$ is one-step strongly linear over $I$. Recall for a subset $A$ of $[n]$, $\begin{pmatrix}A\\i\end{pmatrix}$ denotes the collection of all $i$-subsets of $A$.
  \begin{Proposition} \label{B} Let $I$ be a monomial ideal generated in degree $d$ and $\mathbf{u}$ a  monomial strongly linear over $I$. Assume that $I:\mathbf{u}=(x_i\:\; i\in B)$, where $B\subseteq [n]$.  Let  $A$ be a subset of $[n]$ and denote by $J$ the ideal $I+\mathbf{u}(x_i\:\; i\in A)$.  \begin{enumerate}
      \item For all $s\geq 0$ and for any $\mathbf{a}\in \mathbb{Z}_{\geq 0}^n$  $$\beta_{s,\mathbf{a}}(J)=\beta_{s,\mathbf{a}}(I)+\left\{
                                                 \begin{array}{ll}
                                                 1, & \hbox{$\mathbf{a}$ satifies {\em   ($\spadesuit$);}} \\
                                                 0, & \hbox{otherwise.}
\end{array}\right.$$
                \noindent  Here, {\em  ($\spadesuit$)} means the  conditions that $|\mathbf{a}|=d+s$, $\mathbf{a}-\mathrm{mdeg}(\mathbf{u})\in \{0,1\}^n$ and $\mathrm{supp}(\mathbf{a}-\mathrm{mdeg}(\mathbf{u}))\in \begin{pmatrix}A\cup B\\ s+1\end{pmatrix} \setminus \begin{pmatrix} B\\ s+1\end{pmatrix}$.

      \item For all $s\geq 0$, $$\beta_{s,d+s}(J)=\beta_{s,d+s}(I)+\begin{pmatrix}r_1\\ s+1\end{pmatrix}-\begin{pmatrix}r_2\\ s+1\end{pmatrix},$$ where $r_1=|A\cup B|, r_2=|B|.$
    \end{enumerate}

 \end{Proposition}

 \begin{proof} With the notation as above, as in the proof of  \cite[Proposition  2.4]{BYZ}, we conclude that the connecting homomorphism $\gamma_{i+1,\mathbf{a}}: \mathrm{Tor}_{i+1}(I+L, \mathbb{K})_{\mathbf{a}}\rightarrow  \mathrm{Tor}_{i+1}(I\cap L, \mathbb{K})_{\mathbf{a}}$ induced by the short exact sequence
 $0\rightarrow I\cap L\rightarrow I\oplus L\rightarrow I+L\rightarrow 0$ is surjective for all $i\geq 0$ and all $\mathbf{a}\in \mathbb{Z}_{\geq 0}^n$ with $|\mathbf{a}|>i+d$. Thus, $\mathrm{Tor}_i(I,\mathbb{K})_{\mathbf{a}}\cong \mathrm{Tor}_i(I+L,\mathbb{K})_{\mathbf{a}}$ by \cite[Lemma 2.2]{BYZ}. For $|\mathbf{a}|=i+d$, we have the following exact sequence:
 $$0\rightarrow \mathrm{Tor}_i(I\cap L, \mathbb{K})_{\mathbf{a}}\rightarrow \mathrm{Tor}_i(I,\mathbb{K})_{\mathbf{a}}\oplus \mathrm{Tor}_i(L,\mathbb{K})_{\mathbf{a}}\rightarrow \mathrm{Tor}_i(I+L, \mathbb{K})_{\mathbf{a}}\rightarrow 0.$$
 This yields the required conclusions.
 \end{proof}

\begin{Example} \label{two1} \em Let $I=(x_1x_2x_3, x_3^2x_4,x_3x_4^2)\subseteq \mathbb{K}[x_1,\ldots,x_4]$. Then $x_1x_2$ is strongly linear over $I$. By Proposition~\ref{B},  we have $$\beta_{i,\mathbf{a}}(I+(x_1x_2^2))-\beta_{i,\mathbf{a}}(I)=\left\{
                                                              \begin{array}{ll}
                                                                1, & \hbox{$i=0$ and $\mathbf{a}=(1,2,0,0)$;} \\
                                                                1, & \hbox{$i=1$ and $\mathbf{a}=(2,2,0,0)$;} \\
                                                                0, & \hbox{otherwise.}
                                                              \end{array}
                                                            \right.$$
In particular,  $\mathrm{Projdim}(I+(x_1x_2^2))=\mathrm{Projdim}(I)$.
 \end{Example}

\begin{Corollary} \label{BYZ} Let $I$ and $J$ be monomial ideals generated in degree $d$ such that $I$ is strongly linear over $J$. Then
 $\mathrm{Projdim}(I)\geq \mathrm{Projdim}(J)$.\end{Corollary}

\begin{proof} This  is immediate from Proposition~\ref{B}.
\end{proof}

\section{An application to Squarefree  case}

In this section, we reveal a relationship between a strongly linear monomial over a monomial ideal and a simplicial maximal subcircuit of a uniform clutter, which was introduced in \cite{BYZ}, and then point out some of  main results of \cite{BY} are actually the squarefree variants of results obtained in Section 3.

Recall that a {\it clutter} $\mathcal{C}$ on $[n]$ is a collection of subsets of $[n]$, called {\it circuits} of $\mathcal{C}$, such that if $F_1,F_2$ are distinct circuits  of $\mathcal{C}$, then $F_1\nsubseteq F_2$. A $d$-uniform clutter is a clutter in which every circuit contains exactly $d$ elements. Thus a 2-uniform clutter is nothing but  a simple graph.  Let $\mathcal{C}$ be a $d$-uniform clutter on $[n]$. The {\it circuit ideal} $I(\mathcal{C})$ of $\mathcal{C}$ is defined to be the ideal $$I(\mathcal{C})=(\mathbf{x}_{F}\:\; F\in \mathcal{C}),$$  where $\mathbf{x}_{F}=\prod_{i\in F}x_i$.

   Let $e$ be a $(d-1)$-subset  of $[n]$.  Following \cite{BYZ}, the subset $$N_{\mathcal{C}}[e]:=e\cup \{i\in [n]\:\, e\cup \{i\}\in \mathcal{C}\}$$ is called {\it closed neighborhood} of $e$ in $\mathcal{C}$. We say  $e$ is a {\it simplicial maximal subcircuit} of $\mathcal{C}$ if  $|N_{\mathcal{C}}[e]|\geq d$ and every $d$-subset of $N_C[e]$ is a circuit of $\mathcal{C}$. The set of all simplicial maximal subcircuits of $\mathcal{C}$ is denoted by $\mathrm{Simp}(\mathcal{C})$. We also use $\overline{\mathcal{C}}$ to denote the  complement clutter $\mathcal{C}_n^d\setminus \mathcal{C}$ of  $\mathcal{C}$. Here  $\mathcal{C}_n^d$ is the collection of all $d$-subsets of $[n]$.

Recall from Proposition~\ref{useful} that each monomial which is strongly linear over $I(\overline{\mathcal{C}})$ is squarefree.

\begin{Lemma} \label{clutter} Let $\mathcal{C}$ be a $d$-uniform clutter on $[n]$, and $e$ be  a $(d-1)$-subset of $[n]$. Set $\mathbf{x}_e=\prod_{i\in e}x_i$. Then $e$ is  a simplicial maximal subcircuit of $\mathcal{C}$  if and only if $I(\overline{\mathcal{C}}):\mathbf{x}_{e}$ is generated by a proper subset of  $\{x_i\:\; i\in [n]\setminus e\}$. In particular, if $e\in \mathrm{Simp}(\mathcal{C})$, then $\mathbf{x}_{e}$ is strongly linear over $I(\overline{\mathcal{C}})$.
\end{Lemma}

\begin{proof} First we observe that if $K$ is a subset of $[n]$ with $|K|\geq d$, then $K$ is a clique of $\mathcal{C}$ if and only if for any $F\in \overline{\mathcal{C}}$, $F\setminus K$ is not empty. Note that This statement is immediate from the definitions.

Suppose that $e$ is  a simplicial maximal subcircuit of $\mathcal{C}$. Then $N_{\mathcal{C}}[e]$ is a clique of $\mathcal{C}$ with $|N_{\mathcal{C}}[e]|\geq d$. First, we note that if $i\in [n]\setminus N_{\mathcal{C}}[e]$, then $e\cup \{i\}\notin \mathcal{C}$, and so $\mathbf{x}_{e}x_i\in I(\overline{\mathcal{C}})$. Next, for any $F\in \overline{\mathcal{C}}$, since $F\setminus N_{\mathcal{C}}[e]$ is not empty, we have $\frac{[\mathbf{x}_F,\mathbf{x}_e]}{\mathbf{x}_e}\in (x_i\:\; i\in [n]\setminus N_{\mathcal{C}}[e])$. From this it follows that $I(\overline{\mathcal{C}}):\mathbf{x}_e=(x_i\:\; i\in [n]\setminus N_{\mathcal{C}}[e])$. This proves the necessity.

Conversely, we may assume that $I(\overline{\mathcal{C}}):\mathbf{x}_e=(x_i\:\; i\in A)$ with $A\subsetneq [n]\setminus e$. It is easy to check that $N_{\mathcal{C}}[e]=[n]\setminus A$. Given any $F\in \overline{\mathcal{C}}$, we have $x_k$ divides $\frac{[\mathbf{x}_F,\mathbf{x}_e]}{\mathbf{x}_e}$ for some $k\in A$ by Lemma~\ref{1}. This implies $(F\setminus e)\cap A\neq \emptyset$, and then  $F\setminus ([n]\setminus A)\neq \emptyset$. Hence $N_{\mathcal{C}}[e]$ is a clique of $\mathcal{C}$, and so $e\in \mathrm{Simp}(\mathcal{C})$.
\end{proof}

 By  Lemma~\ref{clutter},   we obtain  the following  squarefree version  of Proposition~\ref{B}.

\begin{Corollary} \label{last} Let $\mathcal{C}$ be a d-uniform clutter on $[n]$ and $e\in \mathrm{Simp}(\mathcal{C})$. Suppose that $A$ is a non-empty subset of $\{F\in \mathcal{C}\:\; e\subsetneq F\}$. Let $\mathcal{D}=\mathcal{C}\setminus A$ and set $I=I(\overline{\mathcal{C}})$ and $J=I(\overline{\mathcal{D}})=I+(\mathbf{x}_{F}\:\; F\in A)$. Put $X=\{i\in [n]: x_i\mathbf{x}_{e}\in A\}$ and $Y=\{i\in [n]\:\; x_i\in I:\mathbf{x}_{e}\}$. Then the following statements hold:
\begin{enumerate}
                       \item For any $i\geq 0$ and any $\mathbf{a}\in \mathbb{Z}^n$, we have  $$\beta_{i,\mathbf{a}}(J)=\beta_{i,\mathbf{a}}(I)+\left\{
                                                 \begin{array}{ll}
                                                 1, & \hbox{$\mathbf{a}$ satifies {\em   ($\clubsuit$);}} \\
                                                 0, & \hbox{otherwise.}
\end{array}\right.$$
                  Here, {\em  ($\clubsuit$)} means the  conditions that $|\mathbf{a}|=d+i$, $\mathbf{a}-\mathrm{mdeg}(\mathbf{x}_e)\in \{0,1\}^n$ and $\mathrm{supp}(\mathbf{a}-\mathrm{mdeg}(\mathbf{x}_e))\in \begin{pmatrix}X\cup Y\\ i+1\end{pmatrix} \setminus \begin{pmatrix} Y\\ i+1\end{pmatrix}$.
                       \item For all $i\geq 0$, $$\beta_{i,d+i}(J)=\beta_{i,d+i}(I)+\begin{pmatrix}r_1\\ i+1\end{pmatrix}-\begin{pmatrix}r_2\\ i+1\end{pmatrix},$$ where $r_1=|X\cup Y|, r_2=|Y|.$
                     \end{enumerate}

\end{Corollary}

It is not difficult to see that  \cite[Theorem 2.4]{BY}  is  a special case of Corollary~\ref{last}.(1) when  $|A|=1$, and \cite[Corollary 2.6]{BY} coincides with Corollary~\ref{last}.(2).

\vspace{5mm}
\noindent{\bf Acknowledgement}:  We thank the anonymous referee for his/her careful reading and useful comments.
This project is supported by NSFC (No. 11971338)

\end{document}